\documentclass[12pt,reqno]{amsart}

\usepackage{graphicx}
\usepackage{mathrsfs}
\usepackage{color}
\usepackage{comment}
\usepackage{amssymb}
\usepackage[normalem]{ulem}

\usepackage{hyperref}

\usepackage[margin = 1.4in] {geometry}

\allowdisplaybreaks

\newtheorem{theorem}[subsection]{Theorem}
\newtheorem{lemma}[subsection]{Lemma}
\newtheorem{cor}[subsection]{Corollary}

\theoremstyle{definition}

\newtheorem{remark}[subsection]{Remark}

\newcommand{\haus}{\mathcal{H}}

\newcommand{\graph}{\mathrm{graph}}

\newcommand{\eps}{\varepsilon}

\newcommand{\R}{\mathbb{R}}

\newcommand{\del}{\partial}

\newcommand{\cA}{\mathcal{A}}
\newcommand{\Lip}{\mathrm{Lip}}

\begin{document}

\title{Weiss monotonicity and capillary hypersurfaces}
\author{Otis Chodosh}
\address{Department of Mathematics, Stanford University, Building 380, Stanford, CA 94305, USA}
\email{ochodosh@stanford.edu}
\author{Nick Edelen}
\address{Department of Mathematics, University of Notre Dame, Notre Dame, IN 46556 USA}
\email{nedelen@nd.edu}
\author{Chao Li}
\address{Courant Institute, New York University, 251 Mercer St, New York, NY 10012, USA}
\email{chaoli@nyu.edu}


\begin{abstract}
Previous work of the authors established the rigorous limiting behavior of minimizing capillary surfaces to minimizers of the Alt--Caffarelli functional as the capillary angle tends to zero. We prove here that in this limit, the capillary area-density converges to the Weiss energy density. We apply this to obtain angle-independent curvature estimates and regularity results for capillary minimizers.
\end{abstract}

\maketitle

\section{Introduction}
We continue to explore the connection between capillary surfaces and the one-phase Bernoulli problem following our work \cite{ChEdLi}. We previously showed that a sequence of smooth capillary minimizers in an Euclidean half-space with small angle will eventually be graphical over the container boundary plane, with the graphing functions subsequentially converging (after renormalization) to a minimizer of the Alt--Caffarelli  (one-phase Bernoulli) functionals as the angle approaches $0$. In \cite{ChEdLi}, this connection was used to establish a classification of minimizing capillary cones with small angle (see also \cite{PaToVe} who considered capillary cones under a positivity assumption). Here, we show that the monotone quantity for capillary surfaces converges (after renormalization) to the Weiss monotonicity formula. We use this to establish  a priori estimates in the spirit of \cite{White}.

Specifically, consider here domains $\Omega \subset \R^{n+1}_+ := \{ x \in \R^{n+1} : x_1 > 0 \}$ which (locally) minimize the capillary functional
\[
\cA^\theta(\Omega) = \haus^n(\del^*\Omega \cap \R^{n+1}_+) - \cos\theta \haus^n(\del^*\Omega \cap \del \R^{n+1}_+)
\]
for a fixed angle $\theta \in (0, \pi)$.  We are concerned primarily with minimizers $\Omega$ which are smooth, that is, domains $\Omega$ for which the interface $M = \del\Omega \cap \R^{n+1}_+$ is a smooth hypersurface in $\R^{n+1}_+$ extending in a smooth fashion to the boundary $\del \R^{n+1}_+$. $M$ is called a capillary hypersurface, and can be thought of as a mathematical model of a fluid interface at equilibrium inside a container.

\vspace{3mm}

Let $\theta_i \to 0_+$, and let us take $\Omega_i$ be a sequence of smooth minimizers of $\cA^{\theta_i}$ in a large ball (say $B_{4}(0)$), with associated capillary hypersurfaces $M_i = \del \Omega_i \cap \R^{n+1}_+$, satisfying $0 \in \del M_i$.  In \cite{ChEdLi} we showed that for $i \gg 1$, there are Lipschitz functions $u_i : B_1^n \to \R$ satsfying in $B_1$:
\begin{equation}\label{eqn:intro-2}
\del \Omega_i \cap \R^{n+1}_+ \subset \graph_{\R^n}(u_i), \quad \del M_i = \del \{ u_i > 0 \}, \quad \sup_i \theta_i^{-1} \Lip(u_i) < c(n),
\end{equation}
here $c(n)$ is a constant that only depends on $n$, and we identify $\R^n \equiv \del \R^{n+1}_+$.  We showed the rescaled functions $\theta_i^{-1}u_i$ converge in $(W^{1,2}_{loc} \cap C^\alpha_{loc})(B_1)$ to some Lipschitz function $v$ which is a minimizer in $B_1^n$ of the Alt-Caffarelli functional
\[J(v) = \int_{\R^n}(|Dv|^2 +1_{\{v>0\}}) dy,\]
and the free-boundaries $\partial M_i$ converges to $\del \{ v > 0 \}$ in the local Hausdorff distance.  The variational problem of $J$ is called the one-phase Bernoulli problem. In low dimensions, or under an a priori bound on curvature like $\sup_i \theta_i^{-1} |A_{M_i}|<\infty$, we showed that the convergence $\theta_i^{-1}u_i \to v$ is in fact $C^{2,\alpha}_{loc}(B_1)$.

The monotonicity formula is a crucial tool to study the variational problem of the capillary functional and the Alt-Caffarelli functional. Let us briefly recall it here. Consider for any of the $\Omega_i$ (being stationary for $\cA^{\theta_i}$ in $B_1$) the associated varifold
\[
V_i = [\partial \Omega_i \cap \R^{n+1}_+] - \cos\theta_i [\partial\Omega_i \cap \del \R^{n+1}_+].
\]
It is not hard to check (see, e.g. \cite{DeMa,DeEdGaLi}) that $V_i$ is a free-boundary stationary varifold in the sense that its first variation vanishes along any compactly supported vector fields that are tangential on $\partial \R^{n+1}_+$. Consequently, for any $x\in \partial \R^{n+1}_+ \cap B_1$, the density ratio
\[\Theta_{V_i}(x,r): = \frac{\|V_i\|(B_r(x))}{\omega_n r^n}\]
is increasing in $r\in (0, 1-|x|)$, here $\omega_n$ is the volume of the unit ball in $\R^{n}$. Moreover, if $x \in \del M_i$ is a regular point of $M_i$ (i.e. $M_i$ is a manifold with smooth boundary near $x$), then the density $\Theta_{V_i}(x):=\lim_{r \to 0} \Theta_{V_i}(x,r) = (1-\cos\theta_i)/2$.

Similarly, there is an important monotonicity property enjoyed by $v$ (see e.g. \cite{Weiss, VeBook}), being a stationary solution of the Alt-Caffarelli functional in $B_1^n$: for any $x\in B_1^n$, the Weiss energy of $v$
\[
W_v(x, r) = r^{-n} \int_{ \{ v > 0 \} \cap B_r^n(x)} (|Dv|^2 + 1) dy - r^{-n-1} \int_{\del B_r^n(x)} v^2 d\sigma
\]
is increasing in $r\in (0,1-|x|)$. If $x \in \del \{ v > 0 \}$ is a regular point, then the limit $W_v(x):= \lim_{r \to 0} W_v(x,r) = \omega_n/2$.


Our main theorem is:
\begin{theorem}[convergence of monotone quantities]\label{thm:mono-conv}
For $\theta_i,\Omega_i,V_i,v$ as above, and any $x_i \in \del \R^{n+1}_+ \to x \in B_1^n$, $r_i \to r \in (0, 1-|x|)$, we have the convergence
\begin{equation}
\theta_i^{-2} \Theta_{V_i}(x_i, r_i) \to \frac{1}{2\omega_n} W_v(x, r). \label{eqn:mono-conv-concl}
\end{equation}
\end{theorem}

Our primary application of Theorem \ref{thm:mono-conv} is the following a priori curvature estimates for capillary minimizing hypersurfaces in all dimensions, assuming that the density is uniformly close to that of a domain enclosed by a flat hyperplane.

\begin{theorem}[a priori curvature estimate]\label{thm:est}
There are constants $\eps(n)$, $c(n)$ so that if $\theta \in (0, \pi)$ and $\Omega$ is a smooth minimizer of $\cA^{\theta}$ in $B_1$ satisfying
\begin{equation}\label{eqn:est-hyp}
\Theta_V(x, r) + (\cos\theta)_- \leq (1+\eps) (1 - \cos\theta)/2 \quad \forall x \in \del M \cap B_1, r \in (0, 1-|x|),
\end{equation}
where $M = \del \Omega \cap \R^{n+1}_+$ and $V = [\del\Omega \cap \R^{n+1}_+] - \cos\theta[ \del \Omega \cap \del \R^{n+1}_+]$, then we have the bound
\begin{equation}
|A_M(x)| \leq c \sin\theta \quad \forall x \in M \cap B_\eps(\del\R^{n+1}_+) \cap B_{1/8}. \label{eqn:est-concl}
\end{equation}
Here $|A_M|$ is the norm of the second fundamental form of $M$.
\end{theorem}

\begin{remark}
The a salient aspect of Theorem \ref{thm:est} is the explicit constant dependencies: the constants $\eps$ and $c$ depend \emph{only} on the ambient dimension. (It would be straightforward to prove a weaker version of Theorem \ref{thm:est} where $\eps$, $c$ depended in addition on $\theta$.)
\end{remark}

\begin{remark}
In \eqref{eqn:est-hyp}, $(\cos\theta)_- \equiv -\min\{ 0, \cos\theta \}$ is the negative part of $\cos\theta$.  The reason for this term (and the $\sin\theta$ in \eqref{eqn:est-concl}) is to leave the hypotheses and conclusions of Theorem \ref{thm:est} unchanged if one replaces $\Omega$ with $\R^{n+1}_+ \setminus \Omega$ and $\theta$ with $\pi - \theta$, an operation which effectively just switches orientation.
\end{remark}

\begin{remark}
If instead of \eqref{eqn:est-hyp} one assumes only a density bound centered at $0$ like
\[
\Theta_V(0, 1) + (\cos\theta)_- \leq (1+\eps/2)(1-\cos\theta)/2 ,
\]
then for a suitable choice of $\delta(n)$, because of the monotonicity formula \eqref{eqn:est-hyp} will hold on the ball $B_\delta$ in place of $B_1$, and consequently Theorem \ref{thm:est} implies
\[
|A_M(x)| \leq c\delta^{-1} \sin\theta \quad \forall x \in M \cap B_{\eps\delta}(\del \R^{n+1}_+) \cap B_{\delta/8}.
\]
\end{remark}

From Theorem \ref{thm:est}, we obtain a Bernstein-type theorem for global minimizers.

\begin{cor}[Bernstein-type theorem]\label{thm:bernstein}
	There is a constant $\eps(n)$ so that if $\theta\in (0,\pi)$ and $\Omega\subset \R^{n+1}_+$ is a smooth minimizer of $\cA^\theta$ in $\R^{n+1}$ satisfying
	\begin{equation}
		\Theta_V(0, \infty) + (\cos\theta)_{-}\le (1+\eps)(1-\cos\theta)/2,
	\end{equation}
	where $V = [\del\Omega \cap \R^{n+1}_+] - \cos \theta [\partial \Omega\cap \partial \R^{n+1}_+]$, then $\Omega$ is the region bounded by a capillary half-plane.
\end{cor}

\begin{proof}
	Take $\eps$ the constant required in Theorem \ref{thm:est}, and write $M = \del \Omega \cap \R^{n+1}_+$. For each $\rho>0$, the rescaled domain $\Omega_\rho:= \rho^{-1}\Omega$ satisfies the assumptions of Theorem \ref{thm:est}. Thus, if $M_\rho := \rho^{-1} M$ we conclude that
	\[|A_{M_\rho}(x)|\le c\sin \theta, \quad \forall x\in M_\rho \cap B_\eps(\R^n)\cap B_{1/8}.\]
	This implies that
	\[|A_M(x)|\le c\rho^{-1} \sin\theta,\quad \forall x\in M\cap B_{\rho\eps}(\R^n)\cap B_{\rho/8}.\]
	Sending $\rho\to \infty$, we have that $|A_M(x)|=0$ for all $x\in M$.
\end{proof}

Another consequence of Theorem \ref{thm:est} is the following regularity result.
\begin{cor}\label{cor:reg}
Taking the graphical functions $u_i$ and limiting function $v$ as above, assume that  the limiting function $v$ is is regular at the free-boundary, then convergence $\theta^{-1}_i u_i \to v$ is $C^{2,\alpha}_{loc}(B_1)$, where we interpret this in the sense of Hodograph transforms near the free-boundary.
\end{cor}

\begin{proof}
By standard estimates the convergence $\theta_i^{-1} u_i \to v$ is smooth away from the free-boundary $\del \{ v > 0 \}$.  If $x \in \del \{ v > 0 \}$ is a regular point, then for any $\eps > 0$ there is a radius $r > 0$ for which $W_v(x, r) \leq (1+\eps/2)\omega_n/2$.  By Theorem \ref{thm:mono-conv}, and the Hausdorff convergence of free-boundaries, we deduce that for a potentially smaller radius $r$, we have
\[
\theta_i^{-2} \Theta_{V_i}(z, s) \leq (1+\eps)(1-\cos\theta_i)/2 \quad \forall z \in \del M_i \cap B_r(x), 0 < s < r.
\]
Now for $\eps(n)$ chosen small, Theorem \ref{thm:est} implies
\[
\sup_i \sup_{M_i \cap B_{r/8}(x) \cap B_{\eps r}(\del \R^{n+1}_+)} \theta_i^{-1} |A_{M_i}| < \infty,
\]
and so the improved convergence of \cite[Proposition 4.11]{ChEdLi} (and the Lipschitz bound \eqref{eqn:intro-2}) implies $\theta_i^{-1} u_i \to v$ in $C^{2,\alpha}$ near $x$, interpreted in the sense of Hodograph transforms.
\end{proof}

O.C. was partially supported by a Terman Fellowship and
an NSF grant (DMS-2304432). N.E. was partially supported by an NSF grant (DMS-2204301).  C.L. was partially supported
by an NSF grant (DMS-2202343), a Simons Junior Faculty Fellowship and a Sloan Fellowship. We are grateful to Rick Schoen for posing a question that inspired this work.

\section{Proof of Main Theorems}

We will identify $\R^n$ with $ \del \R^{n+1} \equiv \{ x_1 = 0 \}$.  Write $\pi : \R^{n+1} \to \R^n$ for the linear orthogonal projection, and $d(x, A) = \inf \{ |x - z| : z \in A \}$ for the Euclidean distance to a set $A$.

\begin{proof}[Proof of Theorem \ref{thm:mono-conv}]
We first regularize the area density and Weiss quantity.  Fix $\zeta(t)$ a smooth, decreasing function which is $\equiv 1$ on $(-\infty, 1-\eps]$ and $\equiv 0$ on $[1, \infty)$.  We define the regularized area density
\[
\Theta^\zeta_{V_i}(x, r) = \frac{1}{\omega_n r^n} \int \zeta(|z - x|/r) d||V_i||(z),
\]
and the regularized Weiss quantity
\[
W^\zeta_v(x, r) = \frac{1}{r^n} \int_{\{ v > 0 \}} \zeta(|y - x|/r) (|Dv|^2 + 1) dy + \frac{1}{r^{n+1}} \int_{\{ v > 0 \}} \zeta'(|y - x|/r) v^2/|y - x| dy.
\]
Though we will not need it here, one can show that both $\Theta^\zeta_{V_i}(x, r)$ and $W^\zeta_v(x, r)$ are increasing in $r$ when $x \in \R^n \equiv \del \R^{n+1}_+$.

It will be important that $\Theta^\zeta$ and $W^\zeta$ appropriately approximate the usual monotone quantities: for $x \in B_1^n \equiv B_1 \cap \del \R^{n+1}_+$, $r \in (0, 1-|x|)$ we have the inequalities
\begin{gather}
(1-\eps)^n \Theta_{V_i}(x, (1-\eps) r) \leq \Theta^\zeta_{V_i}(x, r) \leq \Theta_{V_i}(x, r), \label{eqn:mono-conv1} \\
(1-\eps)^n W_v(x, (1-\eps)r) \leq W^\zeta_v(x, r) \leq W_v(x, r). \label{eqn:mono-conv2}
\end{gather}
The inequalities of \eqref{eqn:mono-conv1} follow trivially from the definition and the structure of $\zeta$.  To see \eqref{eqn:mono-conv2}, without loss of generality set $x = 0$, and then observe
\begin{align*}
W_v^\zeta(0, r) 
&= \frac{1}{r^n} \int_{\{ v > 0 \}} \zeta(|y|/r) (|Dv|^2 + 1) dy + \frac{1}{r^{n+1}} \int_{\R^n} \zeta'(|y|/r) v^2/|y| dy \\
&= \frac{1}{r^n} \int_0^\infty \zeta(s/r) \frac{d}{ds} \left(\int_{\{ v > 0 \} \cap B_s} (|Dv|^2 + 1) dy \right) ds\\
&\qquad  + \frac{1}{r^{n+1}} \int_0^\infty \zeta'(s/r)/s \int_{\del B_s} v^2 d\sigma ds \\
&= \frac{-1}{r^{n+1}} \int_0^\infty \zeta'(s/r) \left( \int_{\{ v > 0 \} \cap B_s} (|Dv|^2 + 1 ) dy - s^{-1} \int_{\del B_s} v^2 d\sigma \right) ds \\
&= \frac{1}{r^{n+1}} \int_0^\infty (-\zeta'(s/r)) s^n W_v(0, s) ds \\
&\leq W_v(0, r) \int_0^\infty (-\zeta'(s/r)) r^{-1} ds \\
&= W_v(0, r),
\end{align*}
having used the monotonicity of $W_v$ and $\zeta$.  By estimating the integrand from below instead (on the set $\{\zeta'\ne 0\}$), we get the opposite bound in \eqref{eqn:mono-conv2}:
\[
W_v^\zeta(0, r) \geq (1-\eps)^n W_v(0, (1-\eps)r) \int_0^\infty (-\zeta'(s/r))r^{-1} ds = (1-\eps)^n W_v(0, (1-\eps)r).
\]

We now claim that
\begin{equation}
\theta_i^{-2} \Theta^\zeta_{V_i}(x_i, r_i) \to \frac{1}{2\omega_n} W^\zeta_v(x, r). \label{eqn:mono-conv3}
\end{equation}
To see this, we note from our hypotheses that$|u_i| + |Du_i| \leq \Gamma \theta_i$ for some uniform $\Gamma$, and recall that $x_i \in \del \R^{n+1}_+ \equiv \R^n$, and then compute
\begin{align*}
&\omega_n r_i^n \Theta^\zeta_{V_i}(x_i, r_i) \\
&= \int_{\{ u_i > 0 \}} \zeta\left(\frac{\sqrt{|y - x_i|^2 + u_i(y)^2}}{r_i}\right) \sqrt{1+|Du_i|^2} dy - \int_{\{ u_i > 0 \}} \zeta(|y - x_i|/r_i) \cos\theta_i dy \\
&= \int_{\{ u_i > 0 \}} \zeta(|y - x_i|/r_i) \left( \sqrt{1+|Du_i|^2} - \cos\theta_i \right) dy \\
&\quad +  \int_{\{ u_i > 0 \}} \left( \zeta\left(\frac{\sqrt{|y - x_i|^2 + u_i(y)^2}}{r_i} \right) - \zeta(|y - x_i|/r_i)\right) \sqrt{1 + |Du_i|^2} dy \\
&= \frac{\theta_i^2}{2} \int_{\{ u_i > 0 \}} \zeta(|y - x_i|/r_i) (\theta_i^{-2} |Du_i|^2 + 1) + O(\theta_i) dy \\
&\quad +  \frac{\theta_i^2}{2} \int_{ \{ u_i > 0 \}} \zeta'(|y - x_i|/r_i) \frac{\theta_i^{-2} u_i^2}{r_i |y - x_i|} + O(\theta_i) dy,
\end{align*}
where we write $f(y) = O(\theta)$ to mean $|f|\leq C(\Gamma,\zeta,r)\theta$.

Now recalling the convergence $\theta_i^{-1} u_i \to v$ in $(C^\alpha_{loc} \cap W^{1,2}_{loc})(B_1)$, and the local Hausdorff convergence $\del \{ u_i > 0 \} \to \del \{ v > 0 \}$ in $B_1$, and (by assumption) our convergence $x_i \to x$, $r_i \to r \in (0, 1-|x|)$, we can take a limit of the above computation as $i \to \infty$ to deduce the asserted convergence \eqref{eqn:mono-conv3}.

By combining \eqref{eqn:mono-conv3} with \eqref{eqn:mono-conv1}, \eqref{eqn:mono-conv2}, we get
\begin{align}
\limsup_i  \theta_i^{-2} \Theta_{V_i}(x_i, r_i) 
&\leq \limsup_i (1-\eps)^{-n} \theta_i^{_2} \Theta_{V_i}^\zeta(x_i, (1-\eps)^{-1} r_i) \nonumber \\
&= \frac{(1-\eps)^{-n}}{2\omega_n} W^\zeta_v(x, (1-\eps)^{-1} r) \nonumber \\
&\leq \frac{(1-\eps)^{-n}}{2\omega_n} W_v(x, (1-\eps)^{-1} r), \label{eqn:mono-conv4}
\end{align}
and similary
\begin{equation}
\liminf_i \theta_i^{-2} \Theta_{V_i}(x_i, r_i) \geq \frac{(1-\eps)^n}{2\omega_n} W_v(x, (1-\eps)r). \label{eqn:mono-conv5}
\end{equation}
Since the function $r \mapsto W_v(x, r)$ is continuous (\cite[Lemma 9.1]{VeBook}), we can take $\eps \to 0$ in \eqref{eqn:mono-conv4}, \eqref{eqn:mono-conv5} to obtain \eqref{eqn:mono-conv-concl}.
\end{proof}

Before proving Theorem \ref{thm:est} we require the following cone classification result, which is essentially contained in \cite[Lemma 4.2]{DeEdGaLi}, but is reproduced here for the convenience of the reader.
\begin{lemma}\label{lem:cone}
Let $\theta \in (0, \pi/2]$, and let $\Omega \subset \R^{n+1}_+$ be a dilation-invariant Caccioppoli set minimizing $\cA^\theta$ in $\R^{n+1}$ which satisfies
\begin{equation}\label{eqn:cone-hyp}
\Theta_V(0) \leq (1-\cos\theta)/2,
\end{equation}
where $V = [\del^* \Omega \cap \R^{n+1}_+] - \cos\theta[\del^*\Omega \cap \del \R^{n+1}_+]$ is the associated capillary varifold.

Then either $\Omega = \emptyset$, $\Omega = \R^{n+1}_+$ (up to measure zero), or $\Omega$ is the region enclosed by a capillary half-plane, i.e. up to rotation and translation in $\del \R^{n+1}_+$,
\begin{equation}\label{eqn:cone-concl}
\Omega = [\{ x_1 > 0, \,\cos\theta x_1 + \sin\theta x_{n+1} < 0 \}] \quad \text{$\haus^{n+1}$-a.e.}
\end{equation}
\end{lemma}

\begin{proof}[Proof of Lemma \ref{lem:cone}]
We prove this by induction on $n$.  When $n=1$, $\Omega$ is enclosed by rays emanating from the origin. If $\Omega \neq \emptyset$ or $\R^2_+$, then the density bound implies that there is only one ray in $\R^2_+$, thus the conclusion follows. Suppose now that $n>1$ and the statement holds for $n-1$. 

Consider $S= \partial^*\Omega\cap \partial \R^{n+1}_+$. By \cite[Theorem 1.10]{DeMa}, $S$ is a set of locally finite perimeter in $\partial \R^{n+1}_+$, and if $\Omega \neq \emptyset, \R^{n+1}_+$ then $\del^* S$ is non-empty. Fix a point $x\in \partial^* S\setminus \{0\}$. By the compactness theorem \cite[Theorem 2.9]{DeMa}, a subsequence of the rescalings $\Omega_r:= (\Omega-x)/r$ converges to a dilation invariant minimizing set $\Omega'$ of $\cA^\theta$, and in this subsequence $[\del^*\Omega_r \cap \R^{n+1}_+]\to [\del^*\Omega' \cap \R^{n+1}_+]$, $[\del^*\Omega_r\cap \partial \R^{n+1}_+]\to [\del^*\Omega'\cap \partial \R^{n+1}_+]$ as varifolds. In particular, writing $V' = [\del^*\Omega'\cap \R^{n+1}_+] - \cos\theta [\del^*\Omega'\cap \partial \R^{n+1}_+]$, we have the density bound $\Theta_{V'}(0)\le (1-\cos\theta)/2$.  By standard splitting using the monotonicity formula \cite[Lemma 2.9]{DeEdGaLi} $\Omega'$ has an additional translational symmetry, so up to rotation can be written $\Omega' = \Omega'' \times \R$ for some $\Omega'' \subset \R^n_+$ minimizing $\cA^\theta$ in $\R^n$.  On the other hand, by our choice of $x$ we necessarily have $\del^*\Omega'' \cap \del \R^n_+$ is a half-hyperplane in $\R^n$ and so by induction we deduce $\Omega'$ is enclosed by a capillary hyperplane.

The above argument and upper-semi-continuity of density shows that $\Theta_V(x) = (1-\cos\theta)/2 = \Theta_V(0)$ for all $x \in \del^*S$.  Since $\haus^{n-1}(\del^* S) > 0$, monotonicity implies $\Omega$ has $(n-1)$-dimensions of translational symmetry, i.e. up to rotation $\Omega = \Omega'''\times \R^{n-1}$.  By the $n=1$ case we deduce $\Omega$ is enclosed by a capillary half-plane.
\end{proof}

\begin{proof}[Proof of Theorem \ref{thm:est}]
Let us remark that it suffices to prove the Theorem with $\theta \in (0, \pi/2]$.  For, if $\theta > \pi/2$ then we can simply replace $\theta$ with $\pi - \theta$ and $\Omega$ with $B_1 \setminus \Omega$, and the hypothesis \eqref{eqn:est-hyp} will continue to hold, and the conclusion \eqref{eqn:est-concl} remains unchanged.

\textbf{Case 1: small angle.}  We first show there is a threshold $\theta_0(n) > 0$ so that \eqref{eqn:est-concl} holds whenever $\theta \in (0, \theta_0)$.  To do this we argue by contradiction: suppose for any fixed $\eps' > 0$, there are sequences $\theta_i \to 0$, $\eps_i \to 0$, minimizers $\Omega_i$ of $\cA^{\theta_i}$ in $B_1$, with associated surfaces $M_i = \del \Omega_i \cap \R^{n+1}_+$ and varifolds $V_i = [M_i] - \cos\theta[\del\Omega_i \cap \del \R^{n+1}_+]$, so that
\begin{equation}\label{eqn:est-1}
\Theta_{V_i}(x, r) \leq (1+\eps_i) (1 - \cos\theta_i)/2 \quad \forall x \in \del M_i \cap B_1, r \in (0, 1 - |x|),
\end{equation}
but for which
\[
\sup_{M_i \cap \{ 0 < x_1 < \eps' \} \cap B_{1/4} } (1/4-|x|) \theta_i^{-1} |A_{M_i}(x)| \to \infty.
\]

By \cite[Lemma 4.10, Lemma 4.13]{ChEdLi}, we can choose (and fix) $\eps'(n)$ sufficiently small so that, when $i \gg 1$, we can find Lipschitz functions $u_i : B_{1/4}^n \to \R$ so that:
\begin{gather*}
M_i \subset \graph_{\R^n}(u_i) \text{ in } B_{1/4} \cap \{ 0 < x_1 < \eps' \}, \quad \del M_i = \del \{ u_i > 0 \} \text{ in } B_{1/4}, \\
\Lip(u_i) \leq c(n) \theta_i.
\end{gather*}

Pick $x_i \in \{ 0 < x_1 < \eps' \} \cap B_{1/4}$ for which
\[
(1/4 - |x_i|) \theta_i^{-1} |A_{M_i}(x_i)| \geq \frac{1}{2} \sup_{M_i \cap \{ 0 < x_1 < \eps' \} \cap B_{1/4} } (1/4 - |x|) \theta_i^{-1} |A_{M_i}(x)|,
\]
and set $\lambda_i = \theta_i^{-1} |A_{M_i}(x_i)|$.  We separate the following two cases.

\textbf{Case 1a:} $\sup_i \lambda_i d(x_i, \del M_i) < \infty$.  Define the rescaled domains $\Omega'_i = \lambda_i(\Omega_i - \pi(x_i))$, surfaces $M_i' = \lambda_i(M_i - \pi(x_i))$, varifolds $V_i' = \lambda_i(V_i - \pi(x_i))$, and points $x_i' = \lambda_i(x_i - \pi(x_i))$.  By our assumption, we can assume $x_i' \to x' \equiv (x'_1, 0)$.

For suitable $R_i \to \infty$, the $\Omega_i'$ are minimizers of $\cA^{\theta_i}$ in $B_{R_i}$ satisfying
\begin{equation}\label{eqn:est-2}
\sup_i d(0, \del M_i') < \infty, \quad \theta_i^{-1} |A_{M_i'}(x_i')| = 1, \quad \sup_{M_i' \cap B_{R_i}} \theta_i^{-1} |A_{M_i'}| \leq 4,
\end{equation}
and additionally
\begin{equation}\label{eqn:est-3}
\theta_i^{-2} \Theta_{V_i}(z, r) \leq (1+\eps_i)\theta_i^{-2} (1-\cos\theta_i)/2 \leq (1+\eps_i)/4
\end{equation}
for every $z \in \del M_i' \cap B_{R_i}$ and every $0 < r < R_i - |z|$.  Moreover, if we let $u_i'(z) = \lambda_i u_i((z - \pi(x_i))/\lambda_i)$, then
\begin{gather*}
M_i' = \graph_{\R^n}(u_i') \text{ in } B_{R_i} \cap \{ x_1 > 0 \}, \quad \del M_i' = \del \{ u_i' > 0 \} \text{ in } B_{R_i}, \\
\Lip(u_i') \leq c(n).
\end{gather*}

We can apply \cite[Proposition 4.11]{ChEdLi} to find a regular, non-zero entire minimizer $v : \R^n \to \R$ of the Alt-Caffarelli functional so that $\theta_i^{-1} u_i' \to v$ in $C^{2,\alpha}_{loc}(\R^n)$, and $\del \{ u_i' > 0 \} \equiv \del M_i' \to \del \{ v > 0 \}$ in the local Hausdorff distance.  From \eqref{eqn:est-3} and Theorem \ref{thm:mono-conv}, we deduce that
\[
W_v(y, r) \leq \omega_n/2 \quad \forall y \in \del \{ v > 0 \}, r > 0.
\]
On the other hand, since $v$ is regular we must have $W_v(y) = \omega_n/2$ at each $y \in \del \{ v > 0 \}$, which implies by the Weiss monotonicity formula that $v(y) = (y\cdot n)_+$ for some unit vector $n$, and hence $|D^2 v| \equiv 0$.  On the other hand, from the improved convergence of \cite[Proposition 4.11]{ChEdLi} and our normalization \eqref{eqn:est-2} we have
\[
1 = \theta_i^{-1}|A_{M_i}(x'_i)| \to |D^2 v(0)|,
\]
which is a contradiction.

\textbf{Case 1b:} $\sup_i \lambda_i d(x_i, \del M_i) = \infty$.  This follows  as in Case 2 of \cite[Lemma 4.14]{ChEdLi}.  We recall the proof below.  Passing to a subsequence we can assume $\lim_i \lambda_i d(x_i, \del M_i) = \infty$.  Define the functions
\[
u_i'(y) = \lambda_i (u_i( (y - \pi(x_i))/\lambda_i) - x_{i, 1}),
\]
where $x_{i, 1}$ is the first coordinate component of $x_i$, so that the surfaces $M_i' = \lambda_i (M_i - x_i)$ are graphs of the $u_i'$.  Then for a suitable $R_i \to \infty$ the $u_i'$ are smooth solutions of the minimal surface equation in $B_{R_i}$ satisfying
\[
\Lip(u_i') \leq c(n) \theta_i, \quad u_i'(0) = 0, \quad \theta_i^{-1}|D^2 u_i'(0)| = 1 + O(\theta_i).
\]

Using standard interior estimates and the structure of the minimal surface equation we can pass to a subsequence, and obtain $C^2_{loc}(\R^n)$ convergence $\theta_i^{-1} u_i' \to v$ for some harmonic $v : \R^n \to \R$ satisfying
\begin{equation}\label{eqn:est-4}
\Lip(v) < \infty,  \quad v(0) = 0, \quad |D^2 v(0)| = 1.
\end{equation}
However the only entire harmonic functions with linear growth are themselves linear, contradicting the last condition of \eqref{eqn:est-4}.

\vspace{3mm}

\textbf{Case 2: large angle.} We now deal with the case when $\theta \in [\theta_0, \pi/2]$.  We claim that for any $\eps' > 0$, provided $\eps(n, \eps')$ is sufficiently small, then we have the bound
\begin{equation}\label{eqn:est-5}
\Theta_V(x, d(x, \del M)/2) \leq 1 + \eps', \quad \forall x \in M \cap \R^{n+1}_+ \cap B_{1/4}.
\end{equation}

We proceed by contradiction.  Suppose otherwise: there are $\eps_i \to 0$, $\theta_i \in [\theta_0, \pi/2]$, smooth minimizers $\Omega_i$ of $\cA^{\theta_i}$ in $B_1$, associated surfaces $M_i = \del \Omega_i \cap \R^{n+1}_+$ and varifolds $V_i = [M_i] - \cos\theta_i[\del\Omega_i \cap \del\R^{n+1}_+]$, so that \eqref{eqn:est-1} holds for all $i$, but for which
\[
\Theta_{V_i}(x_i, d(x_i, \del M_i)/2) > 1+\eps'
\]
for some sequence $x_i \in M_i \cap \R^{n+1}_+ \cap B_{1/4}$.

Set $\lambda_i = d(x_i, \del M_i)^{-1}$, let $z_i \in \del M_i$ realize $d(x_i, \del M_i)$, and define the rescaled domains $\Omega_i' = \lambda_i(\Omega_i - z_i)$, points $x_i' = \lambda_i(x_i - z_i)$.  Then each $\Omega_i'$ is a smooth minimizer of $\cA^{\theta_i}$ in $B_2$, with associated surfaces $M_i' = \lambda_i (M_i - z_i)$, varifolds $V_i' = \lambda_i (V_i - z_i)$, satisfying
\begin{gather*}
0 \in \del M_i', \quad \Theta_{V_i'}(0, 2) \leq (1+\eps_i)(1-\cos\theta_i)/2 , \quad \Theta_{V_i}(x_i', 1/2) \geq 1+\eps'.
\end{gather*}
We can assume $x_i' \to x' \in \del B_{1}$ and $\theta_i \to \theta \in [\theta_0, \pi/2]$.

Passing to a subsequence, the compactness of capillary minimizers (\cite[Theorem 2.9]{DeMa}, \cite[Lemma 3.12]{DeEdGaLi}) implies we can find a domain $\Omega'$ minimizing $\cA^\theta$ in $B_2$, so that $\Omega_i' \to \Omega'$ in $L^1_{loc}$ as Caccioppoli sets, and $V_i' \to V' := [\del^* \Omega' \cap \R^{n+1}_+] - \cos\theta [\del^*\Omega' \cap \del \R^{n+1}_+]$ as varifolds.  Moreover, if we write $S' := \del^*\Omega' \cap \del\R^{n+1}_+$ for the wet region of $\Omega'$, then $S'$ is a set of locally-finite perimeter in $\R^n \cap B_2$, and $\del M_i' \to \overline{\del^*S'}$ in the local Hausdorff distance.  We get
\begin{equation}\label{eqn:est-6}
0 \in \overline{\del^*S'}, \quad \Theta_{V'}(0, 2) \leq (1-\cos\theta)/2, \quad \Theta_{V'}(x', 1/2) \geq 1+\eps'.
\end{equation}

If we take a tangent cone of $\Omega'$ (and $V'$) at $0$, then again from \cite{DeMa}, \cite{DeEdGaLi} we obtain a dilation-invariant minimizer $\Omega''$ of $\cA^\theta$ in $\R^{n+1}$, with associated varifold wet region $S'' = \del^*\Omega'' \cap \del\R^{n+1}_+$, and varifold $V'' = [\del^*\Omega'' \cap \R^{n+1}_+] - \cos\theta[S'']$, with the properties
\[
0 \in \overline{\del^* S''}, \quad \Theta_{V''}(0, \infty) \leq (1-\cos\theta)/2.
\]

By Lemma \ref{lem:cone}, $\Omega''$ must be the capillary half-plane solution.  We deduce $\Theta_{V''}(0) = \Theta_{V'}(0) = (1-\cos\theta)/2$, and therefore by the upper density bound in \eqref{eqn:est-6} the monotonicity formula implies $V'' = V'$.  But now we have $\Theta_{V'}(x', 1/2) \leq 1$, contradicting the lower density bound of \eqref{eqn:est-6}.  This proves \eqref{eqn:est-5}.

\vspace{3mm}

We next claim that if $\eps(n)$ is chosen sufficiently small (and as before $\theta \in [\theta_0, \pi/2]$), then for some constant $c(n)$ we have
\begin{equation}\label{eqn:est-7}
(1/4 - |x|) |A_{M}(x)| \leq c(n) \quad \forall x \in M \cap B_{1/4},
\end{equation}
which will clearly imply \eqref{eqn:mono-conv-concl} since $\theta \geq \theta_0(n)$.

To prove \eqref{eqn:est-7} we again argue by contradiction.  Suppose there are sequences $\eps_i \ll \eps_i' \to 0$, $\theta_i \in [\theta_0, \pi/2]$, minimizers $\Omega_i$ of $\cA^{\theta_i}$ in $B_1$, so that writing $M_i = \del \Omega_i \cap \R^{n+1}_+$ and $V_i = [M_i] - \cos\theta_i [\del\Omega_i \cap \del \R^{n+1}_+]$, we have the bounds \eqref{eqn:est-1} and \eqref{eqn:est-5} (with $V_i, M_i, \eps_i, \eps_i'$ in place of $V, M, \eps, \eps'$), but for which
\[
\sup_{M_i \cap B_{1/4}} (1/4 - |x|) |A_{M_i}(x)| \to \infty.
\]
Choose $x_i \in B_{1/4}$ satisfying
\[
(1/4 - |x_i|) |A_{M_i}(x_i)| \geq \frac{1}{2} \sup_{M_i \cap B_{1/4}} (1/4 - |x|) |A_{M_i}(x)|,
\]
and let $\lambda_i = |A_{M_i}(x_i)|$.  There is no loss in assuming that $\theta_i \to \theta \in [\theta_0, \pi/2]$.  We break into two cases.

\textbf{Case 2a:} $\sup_i \lambda_i d(x_i, \del M_i) < \infty$.  Choose $z_i \in \del M_i$ realizing $d(x_i, \del M_i)$, and defined the rescaled domains $\Omega_i' = \lambda_i(\Omega_i - z_i)$, surfaces $M_i' = \lambda_i(M_i - z_i)$, and points $x_i' = \lambda_i(x_i - z_i) \in M_i'$.  There is no loss in assuming that $x_i' \to x'$.

Then for a suitable sequence $R_i \to \infty$, the $\Omega_i'$ are minimizers of $\cA^{\theta_i}$ in $B_{R_i}(0)$, which satisfy
\[
0 \in \del M_i', \quad |A_{M_i'}(x_i')| = 1, \quad \sup_{M_i' \cap B_{R_i}} |A_{M_i'}| \leq 4,
\]
and
\[
\Theta_{V_i'}(0, R_i) \leq (1+\eps_i)(1-\cos\theta)/2.
\]

Passing to a subsequence, we can apply the compactness of capillary minimizers \cite{DeMa} and standard a priori estimates to find a smooth minimizer $\Omega'$ of $\cA^{\theta}$ in $\R^{n+1}_+$, so that $\Omega_i' \to \Omega'$ in $L^1_{loc}$, and $M_i' \to M' := \del \Omega' \cap \R^{n+1}_+$ smoothly on compact sets, and $V_i' \to V' := [M'] - \cos\theta [\del \Omega' \cap \R^n]$ as varifolds.  In particular, the limit satisfies
\[
0 \in \del M', \quad |A_{M'}(x')| = 1, \quad \Theta_{V'}(0, \infty) \leq (1-\cos\theta)/2.
\]
However, since $0$ is a smooth capillary point, we must have $\Theta_{V'}(0) = (1-\cos\theta)/2$ also, and so by minimal surface monotonicity $M'$ must be a capillary half-plane, contradicting the fact that $|A_{M'}(x')| = 1$.

\textbf{Case 2b:} $\sup_i \lambda_i d(x_i, \del M_i) = \infty$.  Define the rescaled domains $\Omega_i' = \lambda_i(\Omega_i - x_i)$, surfaces $M_i' = \lambda_i(M_i - x_i)$.  Then for $R_i \to \infty$ suitably, the $\Omega_i'$ are sets of least perimeter in $B_{R_i}$, whose boundaries $M_i' = \del \Omega_i'$ in $B_{R_i}$ satisfy
\[
|A_{M_i'}(0)| = 1, \quad \sup_{M_i' \cap B_{R_i}} |A_{M_i'}| \leq 4, \quad \Theta_{M'_i}(0, R_i) \leq 1 + \eps_i'.
\]

Therefore, by compactness of perimeter-minimizing sets and standard a priori estimates we can find a smooth perimeter minimizer $\Omega'$ in $\R^{n+1}$ so that $M_i' \to M' = \del \Omega'$ smoothly on compact sets.  The limit will satisfy
\[
|A_{M'}(0)| = 1, \quad \Theta_{M'}(0, \infty) \leq 1.
\]
However by monotonicity the above implies $M'$ is planar, which is a contradiction.
\end{proof}

\bibliography{bib}
\bibliographystyle{amsplain}

\end{document}